\documentclass[12pt,oneside]{article}
\usepackage{amsmath,amssymb,amsfonts,amsthm}
\usepackage{color}
\usepackage[all]{xy}
\usepackage{graphicx}
\usepackage{color}
\usepackage{makeidx}
\usepackage{multirow}
\usepackage{rotating}
\usepackage{pdfpages}
\usepackage{pdflscape}
\usepackage{tabularray}
\usepackage{multicol}
\usepackage{enumitem}
 \usepackage{hyperref}

\allowdisplaybreaks

\numberwithin{equation}{section} 
\numberwithin{figure}{section} 

\theoremstyle{plain}
\newtheorem*{theorem*}{Theorem}
\newtheorem{theorem}{Theorem}[section]
\newtheorem{lemma}[theorem]{Lemma}
\newtheorem{proposition}[theorem]{Proposition}

\newtheorem{remark}[theorem]{Remark}

\theoremstyle{definition}
\newtheorem{definition}[theorem]{Definition}

\theoremstyle{remark}

\newtheorem*{acknowledgement*}{Acknowledgement}

\numberwithin{equation}{section}

\newcommand\overcirc[1]{\raisebox{10pt}{\tiny{$\circ$}}{\kern-7.5pt}\mbox{$#1$}}
\newcommand\undersym[2]{\raisebox{-6pt}{$#2$}{\kern-5pt}\mbox{$#1$}}
\newcommand\overdiamond[1]{\raisebox{10pt}{\small$\star$}{\kern-7.5pt}\mbox{$#1$}}
\newcommand\overast[1]{\raisebox{10pt}{\small$\ast$}{\kern-7.5pt}\mbox{$#1$}}
\newcommand\overlind[1]{\raisebox{10pt}{\small$\overline{{\hspace{2pt}}\star}$}{\kern-7.5pt}\mbox{$#1$}}
\newcommand\overlinc[1]{\raisebox{10pt}{\small$\overline{{\hspace{2pt}}\circ}$}{\kern-7.5pt}\mbox{$#1$}}
\newcommand\overlina[1]{\raisebox{10pt}{\small$\overline{{\hspace{1pt}}\ast}$}{\kern-7.5pt}\mbox{$#1$}}

\begin{document}

\title{\bf Spherically symmetric Finsler metrics satisfying the $\sigma$T-condition
}

 \author{ Salah G. Elgendi }
\date{}

\maketitle

\begin{center}
  Department of Mathematics, Faculty of Science,
  \\
  Islamic University of Madinah, Madinah, Saudi Arabia
\end{center}

\begin{center}
     salah.ali@fsc.bu.edu.eg, \ salahelgendi@yahoo.com
\end{center}

\vspace{0.3cm}

\begin{abstract}
In this paper, we present a complete characterization of spherically symmetric Finsler metrics that satisfy the $\sigma T$-condition. 
We further investigate the subclass of such metrics within the Landsberg category and identify the precise conditions under which spherically symmetric Finsler metrics satisfying the $T$-condition must also be Berwaldian. 
In addition, we construct new non-regular solutions to the classical unicorn problem, providing explicit examples of Landsberg metrics that are not Berwaldian.
\end{abstract}

 \noindent{\bf Keywords:\/}\,   Spherically symmetric metrics; T-tensor;   T-condition; $\sigma T$-condition; Berwald spaces; Landsberg spaces.

\medskip\noindent{\bf MSC 2020:\/}  53B40; 53C60.


\section{Introduction}

An important class of Finsler metrics is that of \emph{spherically symmetric metrics}, invariant under the action of the rotation group $O(n)$.  
These metrics are defined on an open ball $\mathbb{B}^n(r_0) \subset \mathbb{R}^n$ and take the form
\[
F(x, y) = u\,\phi(r, s), \qquad r = |x|,\quad s = \frac{\langle x, y \rangle}{u}, \quad u=|y|,
\]
where $\phi(r, s)$ is a smooth function of the radial coordinate $r$ and the directional variable $s$.  
Spherically symmetric Finsler metrics have been extensively studied due to their elegant structure and tractable curvature properties \cite{Mo2006,Shen2003}.  
They generalize rotational symmetry in Riemannian geometry and provide a natural framework for studying special classes of Finsler metrics such as Berwald,  and Landsberg metrics.  

Beyond their mathematical significance, spherically symmetric metrics appear in several physical contexts, especially in Finslerian extensions of gravitational and cosmological models \cite{Rutz1993,Pfeifer2012}.  
Their symmetry makes explicit computations feasible and provides natural models for investigating curvature structures and variational principles in Finsler geometry \cite{Zhou_Mo}.  

\medskip

Among the fundamental tensorial objects in Finsler geometry is the \emph{T-tensor}, introduced by Matsumoto \cite{ttensor}.  
The T-tensor plays a significant role in the study of Finsler spaces. In particular, metrics satisfying the \emph{T-condition} ($T=0$), as well as those considered within the broader framework of the \emph{$\sigma T$-condition}, exhibit distinctive geometric properties.

From both mathematical and physical perspectives, the T-tensor remains a fundamental non-Riemannian quantity.  
It naturally arises in the decomposition of the curvature tensor, contributes to the classification of Landsberg and Berwald spaces \cite{Elgendi-LBp,Elgendi2021,Elgendi-ST_condition}, and even plays a role in Finslerian generalizations of general relativity, where it governs anisotropic corrections in geodesic deviation and field equations \cite{Asanov1985,Vacaru2012}.  
Despite its importance, explicit computations of the T-tensor are rare and typically restricted to highly symmetric situations.  

\medskip

In \cite{Elgendi-T-tensor-2025}, a full characterization of spherically symmetric Finsler metrics satisfying the T-condition was obtained for dimensions $n \geq 3$, showing that such metrics must take the form
\[
\phi(r,s) =
a(r)\, s^{\frac{c(r)\, r^2 - 1}{c(r)\, r^2}} \left( r^2 -  s^2 \right)^{\frac{1}{2 c(r) r^2}},
\]
where $a(r)$ and $c(r)$ are smooth functions of the radial variable $r$.  

\bigskip

The present paper extends this study to the \emph{$\sigma T$-condition}.  
Our main goals are as follows:  
\begin{itemize}
    \item[(i)] to characterize all spherically symmetric Finsler metrics satisfying the $\sigma T$-condition;  
    \item[(ii)] to determine when a spherically symmetric metric satisfying the T-condition must be Berwaldian;  
    \item[(iii)] to investigate when metrics satisfying the $\sigma T$-condition belong to the Landsberg class but are not Berwaldian, thereby producing new solutions to the unicorn Landsberg problem.  
\end{itemize}

We prove that a spherically symmetric Finsler metric $F=u\phi(r,s)$ satisfies the $\sigma T$-condition if and only if
\begin{align*}
\phi(r,s)
&=\exp\left(\int \frac{c_1 s+c_2 \sqrt{r^2-s^2}}{r^2+c_1 s^2+c_2 s \sqrt{r^2-s^2}}\,ds\right) \\[4pt]
&=A(r)\, \sqrt{r^2+c_1 s^2+c_2 s \sqrt{r^2-s^2}}\,
   \exp\!\left(
      \frac{c_2}{\sqrt{c_2^2-4c_1-4}}
      \operatorname{arctanh}\!\left(
         \frac{c_2s+2\sqrt{r^2-s^2}}{s\sqrt{c_2^2-4c_1-4}}
      \right)
   \right),
\end{align*}
where $A(r)$, $c_1(r)$, and $c_2(r)$ are smooth functions of $r$.  

\medskip

By \cite{Elgendi-LBp}, a Landsberg metric $F$ remains Landsberg under the conformal change $F \mapsto e^{\sigma}F$ if and only if the $\sigma T$-condition is satisfied.  
Motivated by this, we study when a spherically symmetric metric satisfying the $\sigma T$-condition can yield new examples of non-regular Landsberg metrics that are not Berwaldian.  
In particular, we obtain two families of such metrics (together with their conformal deformations), given explicitly by
\[
\phi(r,s) = A(r)\,
\sqrt{\,r^2 + c_1 s^2 + c_2 s \sqrt{r^2 - s^2}\,}\,
\exp\!\left(
\frac{c_2}{\sqrt{c_2^2 - 4c_1 - 4}}
\operatorname{arctanh}\!\left(
\frac{c_2 s + 2\sqrt{r^2 - s^2}}{s\sqrt{c_2^2 - 4c_1 - 4}}
\right)
\right),
\]
and
\[
\phi(r,s) = A(r)\,
\sqrt{r^2+(bc_2^2-1) s^2+c_2 s \sqrt{r^2-s^2}}\,
\exp\!\left(
\frac{1}{\sqrt{1-4b}}
\operatorname{arctanh}\!\left(
\frac{c_2s+2\sqrt{r^2-s^2}}{c_2s\sqrt{1-4b}}
\right)
\right),
\]
under the restrictions stated in Theorems \ref{Thm:Landsberg_1} and \ref{Thm:Landsberg_2}.  
These provide new explicit solutions to the unicorn Landsberg problem.

\section{Preliminaries} 
~\par 
 
Spherically symmetric Finsler metrics form an important class of Finsler structures characterized by rotational invariance around the origin. Defined on an open ball $\mathbb{B}^n(r_0) \subset \mathbb{R}^n$, such a metric takes the form
\[
F(x, y) = u\, \phi(r, s), \quad \text{where} \quad r = |x|,\; u = |y|,\; s = \frac{\langle x, y \rangle}{|y|},
\]
where $\phi$ is a smooth, positive function defined on $[0, r_0) \times \mathbb{R}$, and $\langle \cdot, \cdot \rangle$ and $|\cdot|$ denote the standard Euclidean inner product and norm, respectively.
\begin{equation*}
\label{Regular_condition}
\phi - s \phi_s > 0, \quad \phi - s \phi_s + (r^2 - s^2)\phi_{ss} > 0 \quad \text{for } n \geq 3,
\end{equation*}
or
\[
\phi - s \phi_s + (r^2 - s^2)\phi_{ss} > 0 \quad \text{for } n = 2,
\]
for all \(|s| \leq r < r_0\). Here, \(\phi_s\) and \(\phi_{ss}\) denote the first and second partial derivatives of \(\phi\) with respect to \(s\), respectively.

These metrics were studied extensively in \cite{MoZhou2009, MoZhouZhu2010}, where their geometric and curvature properties were thoroughly analyzed. The spherical symmetry ensures invariance under the orthogonal group $O(n)$, which allows significant simplification in the study of curvature conditions.

Spherically symmetric metrics generalize radial Riemannian structures and provide a rich framework for constructing explicit examples of Finsler metrics with special geometric properties, such as Landsberg or Berwald metrics \cite{Elgendi2021-SSM,Elgendi2023-SSM,Zhou_Mo,Tayebi-et.}. Owing to their symmetry, many geometric quantities associated with these metrics depend only on the variables $r$ and $s$.

\bigskip

Using the Euclidean metric tensor components $\delta_{ij}$, we may lower the indices of $x^i$ and $y^i$ via
\[
x_i := \delta_{ih} x^h, \quad y_i := \delta_{ih} y^h.
\]
That is, we have $x_i = x^i$ and $y_i = y^i$. It is important to emphasize that in this context,
\[
y_i \neq F \frac{\partial F}{\partial y^i}, \quad \text{but rather} \quad y_i = u \frac{\partial u}{\partial y^i}.
\]

\bigskip

The components of the metric tensor $g_{ij}$ associated with the spherically symmetric Finsler metric $F = u \phi(r,s)$ are given by
\begin{equation}\label{Eq:g_ij}
g_{ij} = \sigma_0\, \delta_{ij} + \sigma_1\, x_i x_j + \frac{\sigma_2}{u} (x_i y_j + x_j y_i) + \frac{\sigma_3}{u^2} y_i y_j,
\end{equation}
where
\[
\sigma_0 = \phi(\phi - s \phi_s), \quad \sigma_1 = \phi_s^2 + \phi \phi_{ss}, \quad \sigma_2 = (\phi - s \phi_s)\phi_s - s \phi \phi_{ss}, \quad \sigma_3 = s^2 \phi \phi_{ss} - s(\phi - s \phi_s)\phi_s.
\]

Throughout this work, subscripts such as \(\phi_s\) denote partial derivatives with respect to the variable \(s\).

The inverse metric tensor components $g^{ij}$ are given by (cf. \cite{Zhou_Mo}):
\begin{equation}\label{Eq:g^ij}
g^{ij} = \rho_0\, \delta^{ij} + \frac{\rho_1}{u^2} y^i y^j + \frac{\rho_2}{u} (x^i y^j + x^j y^i) + \rho_3\, x^i x^j,
\end{equation}
where
\begin{align*}
\rho_0 &= \frac{1}{\phi(\phi - s \phi_s)}, \\
\rho_1 &= \frac{(s\phi + (r^2 - s^2)\phi_s)(\phi \phi_s - s \phi_s^2 - s \phi \phi_{ss})}{\phi^3(\phi - s \phi_s)(\phi - s \phi_s + (r^2 - s^2)\phi_{ss})}, \\
\rho_2 &= -\frac{\phi \phi_s - s \phi_s^2 - s \phi \phi_{ss}}{\phi^2(\phi - s \phi_s)(\phi - s \phi_s + (r^2 - s^2)\phi_{ss})}, \\
\rho_3 &= -\frac{\phi_{ss}}{\phi(\phi - s \phi_s)(\phi - s \phi_s + (r^2 - s^2)\phi_{ss})}.
\end{align*}

\medskip

The \emph{geodesic spray coefficients} $G^i$ for the spherically symmetric Finsler metric $F$ are given by:
\begin{equation} \label{G}
G^i = u P y^i + u^2 Q x^i,
\end{equation}
where $P = P(r, s)$ and $Q = Q(r, s)$ are defined by
\begin{equation} \label{P,Q}
Q := \frac{1}{2r} \cdot \frac{-\phi_r + s \phi_{rs} + r \phi_{ss}}{\phi - s \phi_s + (r^2 - s^2) \phi_{ss}}, \quad
P := -\frac{Q}{\phi} \left( s \phi + (r^2 - s^2) \phi_s \right) + \frac{1}{2r\phi} \left( s \phi_r + r \phi_s \right).
\end{equation}

\medskip

\begin{proposition}[\cite{Elgendi2021-SSM}]
Let $S$ be a spray of the form \eqref{G} with arbitrary functions $P(r,s)$ and $Q(r,s)$. Then $S$ is metrizable by a Finsler metric $F = u\phi(r,s)$ if and only if $\phi$ satisfies:
\begin{equation} \label{Comp_C_C_2}
\begin{split}
C_1 &:= \left(1 + sP - (r^2 - s^2)(2Q - sQ_s)\right)\phi_s + \left(sP_s - 2P - s(2Q - sQ_s)\right)\phi = 0, \\
C_2 &:= \frac{1}{r} \phi_r - (P + Q_s(r^2 - s^2))\phi_s - (P_s + sQ_s)\phi = 0.
\end{split}
\end{equation}
\end{proposition}

\medskip

\begin{definition}
A Finsler manifold $(M, F)$ is called:
\begin{itemize}
  \item \emph{Berwald} if its Berwald tensor $G^i_{jkh} :=\frac{\partial^3G^i}{\partial y^j \partial y^k \partial y^h} 0$ vanishes, that is, $G^i_{jkh} = 0$,
  \item \emph{Landsberg} if its Landsberg tensor $L_{jkh} := -\frac{1}{2} F G^i_{jkh} \frac{\partial F}{\partial y^i} $ vanishes, that is, $L_{jkh} = 0$.
\end{itemize}
\end{definition}

\begin{theorem}[\cite{Elgendi2021-SSM}] \label{Theorem_A}
Let $F = u \phi(r,s)$ be a spherically symmetric Landsberg Finsler metric in dimension $n \geq 3$. Then $F$ is either Riemannian, or its geodesic spray is given by:
\begin{equation} \label{Zhou_P&Q}
P = c_1 s + \frac{c_2}{r^2} \sqrt{r^2 - s^2}, \quad Q = \frac{1}{2} c_0 s^2 - \frac{c_2 s}{r^4} \sqrt{r^2 - s^2} + c_3,
\end{equation}
where $c_0, c_1, c_2, c_3$ are arbitrary smooth functions of $r$.
\end{theorem}

According to \cite{Elgendi2021-SSM}, the Landsberg condition in dimension $n \geq 3$ is equivalent to the following:
\begin{equation} \label{Zhou_L_1_2}
\begin{split}
L_1 &= 3\phi_s P_{ss} + \phi P_{sss} + \left(s\phi + (r^2 - s^2)\phi_s\right) Q_{sss} = 0, \\
L_2 &= -s\phi P_{ss} + \phi_s (P - sP_s) + \left(s\phi + (r^2 - s^2)\phi_s\right)(Q_s - sQ_{ss}) = 0.
\end{split}
\end{equation}

Moreover, the Berwald condition in dimension $n \geq 3$ is equivalent to:
\begin{equation}\label{Berwald_conditions}
P - sP_s = 0, \quad Q_s - sQ_{ss} = 0.
\end{equation}

\medskip

To proceed with the investigation of the T-tensor, we introduce the covector
\[
m_i = x_i - \frac{s}{u} y_i,
\]
which plays a central role in various geometric identities.

Define the angular metric  $\hbar_{ij} $ associated with the Euclidean metric by
\[
\hbar_{ij} := \delta_{ij} - \frac{1}{u^2} y_i y_j.
\]

We also adopt the notations \( m^2 := r^2 - s^2 \), \( \mu := \sigma_1 \), and use:
\[
\mu_s = (\sigma_1)_s, \quad \mu_{ss} = (\sigma_1)_{ss}.
\]

The following identities hold:
\begin{align}
(\sigma_0)_s &= \sigma_2, & (\sigma_2)_s &= -s \mu_s, \notag\\
(\sigma_3)_s &= s^2 \mu_s - \sigma_2, & \sigma_3 &= -s \sigma_2, \label{Identites_sigmas}
\end{align}
\begin{align}
m^2 &= m^i m_i = x^i m_i, & x^r \hbar_{ri} &= m_i, & y^i m_i &= 0, & y_i m^i &= 0, \label{Identites_m^2}
\end{align}
\begin{equation}
\frac{\partial}{\partial y^k} \left( \frac{y_i}{u} \right) = \frac{1}{u} \left( \delta_{ik} - \frac{1}{u^2} y_i y_k \right) = \frac{1}{u} \hbar_{ik}. \label{Identites_hbar}
\end{equation}

\begin{proposition}\cite{Elgendi-T-tensor-2025}\label{Lem:T-tensor_SSFM}
The T-tensor of a spherically symmetric Finsler metric $F = u \phi(r,s)$ is given by:
\begin{align*}
T_{hijk} &= \Phi (\hbar_{hi}\hbar_{jk}+\hbar_{hj}\hbar_{ik}+\hbar_{hk}\hbar_{ij}) \\
&\quad + \Psi (\hbar_{hk}m_im_j+\hbar_{hj}m_im_k+\hbar_{hi}m_jm_k+\hbar_{ij}m_hm_k+\hbar_{jk}m_im_h+\hbar_{ik}m_jm_h) \\
&\quad + \Omega\, m_h m_i m_j m_k,
\end{align*}
where
\begin{align}
\Phi &= -\frac{\phi \sigma_2}{4u} \left(2 s + \sigma_2 m^2 \kappa \right), \label{Eq:PhiFactored}\\
\Psi &= \frac{\phi}{4u} \left( \frac{4\phi_s \sigma_2}{\phi} - 2 s \mu_s - 2 \rho_0 \sigma_2^2 - \sigma_2 \kappa (2 \sigma_2 + \mu_s m^2) \right), \label{Eq:PsiFactored}\\
\Omega &= \frac{\phi}{4u} \left( \frac{8\mu_s \phi_s}{\phi} + 2 \mu_{ss} - 6 \rho_0 \sigma_2 \mu_s - 3 (2 \sigma_2 + \mu_s m^2)(\kappa \mu_s + 2 \rho_3 \sigma_2) \right), \label{Eq:OmegaFactored}
\end{align}
and \(\kappa := \rho_0 + \rho_3 m^2\).
\end{proposition}

Define
\[
W(r,s) := \frac{\phi_s}{\phi - s \phi_s},
\]
and observe that
\begin{equation} \label{EQ:W_1}
\phi(r,s) = \exp\left( \int_0^s \frac{W}{1 + sW} \, ds \right).
\end{equation}

\begin{proposition}\cite{Elgendi-T-tensor-2025}
\label{Prop:Phi=0}
The following identity holds:
\[
2s + m^2 \sigma_2 \kappa = \frac{(\phi - s \phi_s)^2}{\phi(\phi - s \phi_s + m^2 \phi_{ss})} \left( W_s + \left( \frac{1}{s} + \frac{2s}{m^2} \right) W + \frac{2}{m^2} \right).
\]
\end{proposition}

\begin{theorem}\cite{Elgendi-T-tensor-2025}
\label{Theorem_T-condition}
A spherically symmetric Finsler metric $F = u \phi(r,s)$ with $n \geq 3$ satisfies the T-condition (i.e., $T_{hijk} = 0$) if and only if either $F$ is Riemannian, or
\begin{equation} \label{berwald}
\phi(r,s) = a(r)\, s^{\frac{c(r) r^2 - 1}{c(r) r^2}} (r^2 - s^2)^{\frac{1}{2 c(r) r^2}},
\end{equation}
for some smooth functions $a(r)$ and $c(r)$.
\end{theorem}

 \section{T-condition and $\sigma$T-condition}

  This section focuses on determining when metrics with the $T$-condition are Berwaldian, and on characterizing the metrics that satisfy the $\sigma T$-condition.
We start with the following theorem.

 \begin{theorem} 
A spherically symmetric Finsler metric $F=u\phi(r,s)$ with $n \geq 3$ satisfies the $T$-condition is Berwaldian if and only if $c(r)=\frac{C}{r^2}$, where $C$ is a constant not equal to $1$.
\end{theorem}
 \begin{proof}
 By performing explicit calculations, for instance with the aid of \textsc{Maple}, and using the functions $P$ and $Q$ given in \eqref{P,Q} and \eqref{berwald}, we obtain
  $$P=\frac{sc'}{2cr} , \quad c':=\frac{dc}{dr} ,   $$
  $$ Q=\frac{\left(2cr^3(a'c - c' a ) -a  (rc'  +2c )\ln(r^2-s^2)  +2 a  (r   c' +2c)\ln(s)\right)s^2   +2r^2ac (cr^2-1)}{4a c r^4(cr^2-1)}.$$
  Moreover, we have 
  $$ P-sP_s=0, \quad  Q_s-sQ_{ss} =-\frac{(rc'+2c)s}{  c( r^2-s^2)^2(cr^2-1)}.$$
  For $n \geq 3$, $F$  is Berwaldian if and only if $$ P-sP_s=0, \quad  Q_s-sQ_{ss}=0.$$
  So,  $F$  is Berwaldian if and only if the ODE  $rc'+2c=0$ is satisfied.  Which has  the solution
  $$c=\frac{C}{r^2},$$
  where $C$ is a constant with $C\neq 1$.
 \end{proof} 
 
\begin{proposition}\label{raised-tensor} The T-tensor $T^h_{ijk}:=g^{hr}T_{rijk}$ of a spherically symmetric Finsler metric $F=u\phi(r,s)$ is given by
\begin{eqnarray*}
T^h_{ijk}&=&\rho_0\Phi (\hbar^h_{i}\hbar_{jk}+\hbar^h_{j}\hbar_{ik}+\hbar^h_{k}\hbar_{ij})+\rho_0\Psi (\hbar^h_{k}m_im_j+\hbar^h_{j}m_im_k+\hbar^h_{i}m_jm_k+\hbar_{ij}m^hm_k\\
&&+\hbar_{jk}m_im^h+\hbar_{ik}m_jm^h)+\rho_0\Omega  m^hm_im_jm_k+(\rho_3 x^h+\frac{\rho_2}{u} y^h)(\Phi(\hbar_{ik}m_{j}+\hbar_{ij}m_{k}+\hbar_{jk}m_{i})\\
&&+\Psi(m^2(\hbar_{ik}m_{j}+\hbar_{ij}m_{k}+\hbar_{jk}m_{i})+3m_im_jm_k)+\Omega m^2 m_im_jm_k).
\end{eqnarray*}
\end{proposition}
\begin{proof}
The proof is a straightforward calculations by using \eqref{Eq:g^ij}, \eqref{Identites_m^2}, and  Proposition \ref{Lem:T-tensor_SSFM}.
\end{proof}

\begin{theorem}\label{sigma$T$-condition}
A spherically symmetric metric $F=u\phi(r,s)$ with $n\geq 3$ satisfies the $\sigma$T-condition if and only if
\begin{itemize}
\item[(a)] $\Phi+m^2\Psi=0$,

\item[(b)] $m^2\Omega+3 \Psi=0$.
\end{itemize}
\end{theorem}

\begin{proof}  Since the position on the base manifold is determined by the argument $r$, we may, without loss of generality, assume that $\sigma = \sigma(r)$.  
Then, by Proposition~\ref{raised-tensor}, we obtain
\[
\sigma_h := \partial_h \sigma 
= \sigma' \, \frac{x_h}{r}, \quad \sigma':=\frac{d\sigma}{dr}.
\]
Now, we have
\begin{eqnarray*}
\sigma_h T^h_{ijk}&=& \frac{\sigma_r}{r} x_h  T^h_{ijk}\\
                  &=& \frac{\sigma_r}{r}  \Big{(} \rho_0 \Phi (\hbar_{ik}m_{j}+\hbar_{ij}m_{k}+\hbar_{jk}m_{i})+\rho_0 \Psi(3m_im_jm_k+m^2(\hbar_{ik}m_{j}+\hbar_{ij}m_{k}+\hbar_{jk}m_{i}))\\
                  &&+ \rho_0\Omega m^2 m_im_jm_k+(s\rho_2+r^2\rho_3)(\Phi(\hbar_{ik}m_{j}+\hbar_{ij}m_{k}+\hbar_{jk}m_{i})+\Omega m^2 m_im_jm_k\\
                  &&+\Psi (m^2(\hbar_{ik}m_{j}+\hbar_{ij}m_{k}+\hbar_{jk}m_{i})+3m_im_jm_k\Big{)}\\
                  &=& \frac{\sigma_r}{r} (\rho_0+s\rho_2+r^2\rho_3) \Big{(} (\Phi+m^2\Psi)(\hbar_{ik}m_{j}+\hbar_{ij}m_{k}+\hbar_{jk}m_{i})+(\Psi+m^2\Omega)m_im_jm_k\Big{)}
\end{eqnarray*}
 By setting $\sigma_h T^h_{ijk}=0$. For, $n\geq 3$, one can see that the property 
 $$\lambda   (\hbar_{ik}m_{j}+\hbar_{ij}m_{k}+\hbar_{jk}m_{i})+\mu m_im_jm_k=0$$
 implies that $\lambda=0$ and $\mu=0$, where $\lambda$ and $\mu$ are scalar functions on $T\mathbb{B}(r_0)$. This completes the proof. 
\end{proof}

\begin{theorem}
A spherically symmetric metric $F=u\phi(r,s)$     with $n\geq 3$ satisfies the $\sigma$T-condition if and only if it satisfies the $T$-condition or $\phi$ is given by
\begin{equation}\label{Landsberg}
\phi(r, s) = A(r) \cdot \exp\left( \int \frac{c_1(r)s + c_2(r)\sqrt{r^2 - s^2}}{1 + c_1(r)s^2 + c_2(r)s\sqrt{r^2 - s^2}} \, ds \right)
\end{equation}
\end{theorem}
\begin{proof} Let  $F=u\phi(r,s)$  be a spherically symmetric metric satisfying the     $\sigma$T-condition. 
Using the facts that 
$$W=\frac{\phi_s}{\phi-s\phi_s}, \quad W_s=\frac{\phi\phi_{ss}}{(\phi-s\phi_s)^2},$$
one can see that
$$\sigma_2=(W-sW_s)(\phi-s\phi_s)^2.$$
Therefore,  we can write $\Phi$ and $\Psi$ in terms of $W(r,s)$ and its derivations with respect to $s$, as  follows:
$$\Phi=-\frac{\phi (\phi-s\phi_s)^2(W-sW_s)(sm^2 \phi_s W_s+(2s\phi+m^2\phi_s)W)}{4 u(m^2 \phi_s W_s+\phi W)},$$
$$\Psi=-\frac{\phi (\phi-s\phi_s)^2W_{ss}(sm^2 \phi_s W_s+(2s\phi+m^2\phi_s)W)}{4 u(m^2 \phi_s W_s+\phi W)}.$$

Now,  using the fact that $\phi-s\phi_s\neq 0$, the condition $\Phi+m^2\Psi=0$ gives the following two possible PDEs
\begin{equation}\label{nontrivial_case}
(r^2-s^2)W_{SS}-sW_s+W=0
\end{equation}
or
\begin{equation}\label{trivial_case}
s(r^2-s^2) \phi_s W_s+2s\phi W+m^2\phi_sW=0
\end{equation}
The PDE \eqref{trivial_case}, using the fact that $W(\phi-s\phi_s)=\phi_s$, can be rewritten in the form
$$W_s+\left(\frac{1}{s}+\frac{2s}{m^2}\right)W=-\frac{2}{m^2}$$
which gives the trivial case, that is, the T-tensor vanishes. The PDE \eqref{nontrivial_case} has the solution
$$W(r,s)=c_1(r)s+c_2(r)\sqrt{r^2-s^2}.$$
By using \eqref{EQ:W_1}, $\phi(r,s)$ is given by   

\[
\phi(r, s) = A(r) \cdot \exp\left( \int \frac{c_1(r)s +c_2(r)\sqrt{r^2 - s^2}}{1 +c_1(r)s^2 + c_2(r)s\sqrt{r^2 - s^2}} \, ds \right).
\]
To check the second condition $m^2\Omega+3 \Psi=0$ in Theorem \ref{sigma$T$-condition}, we can follow the following strategy to avoid complications. 
Using the notations $T_s:=\frac{\phi_s}{\phi}$ and $T_{ss}:=\frac{\phi_{ss}}{\phi}$, we have:
$$T_s=\frac{\phi_s}{\phi}=\frac{c_1 s+c_2 \sqrt{r^2-s^2}}{1+c_1 s^2+c_2 s\sqrt{r^2-s^2}},$$
$$T_{ss}=\frac{\partial T_s}{\partial s}+T_s^2=\frac{\phi_{ss}}{\phi}=\frac{-c_2 s+c_1 \sqrt{r^2-s^2}}{(1+c_1 s^2+c_2 s\sqrt{r^2-s^2})^2\sqrt{r^2-s^2} }.$$
Moreover, the quantities $\sigma$'s and $\rho$'s can be rewritten in the following form:
\[
\sigma_0 = \phi^2(1 - s T_s), \quad \sigma_1 =\phi^2( T_s^2 + T_{ss}), \quad \sigma_2 = \phi^2((1 - s T_s)T_s - s T_{ss}), \quad \sigma_3 =\phi^2( s^2 T_{ss} - s(1 - s T_s)T_s),
\]
\begin{align*}
\rho_0 &= \frac{1}{\phi^2(1 - s T_s)}, \\
\rho_1 &= \frac{(s  + (r^2 - s^2)T_s)(T_s - s T_s^2 - s T_{ss})}{\phi^2(1 - s T_s)(1 - s T_s + (r^2 - s^2)T_{ss})}, \\
\rho_2 &= -\frac{T_s - s T_s^2 - s T_{ss}}{\phi^2(1 - s T_s)(1- s T_s + (r^2 - s^2)T_{ss})}, \\
\rho_3 &= -\frac{T_{ss}}{\phi^2(1 - s T_s)(1 - s T_s + (r^2 - s^2)T_{ss})}.
\end{align*}
Therefore, we have
$$\mu=\sigma_1=\phi^2( T_s^2 + T_{ss}),$$
$$\mu_s=\phi^2(2T_s(T_s^2+T_{ss})+(T_s^2+T_{ss})_s)$$
$$\mu_{ss}=\phi^2(2T_s(2T_s(T_s^2+T_{ss})+(T_s^2+T_{ss})_s)+(2T_s(T_s^2+T_{ss})+(T_s^2+T_{ss})_s)_s).$$
One may use \textsc{Maple} to compute $\Phi$, $\Psi$, and $\Omega$ as follows:
$$\Phi=-\frac{r^2c_2\phi^3(2(c_1r^2+1)s+c_2r^2 \sqrt{r^2-s^2})}{4u(c_1r^2+1)\sqrt{r^2-s^2}(1+c_1s^2+c_2s\sqrt{r^2-s^2})^2},$$
\begin{align*}
\Psi=&\frac{r^2c_2\phi^3}{4u(c_1r^2+1)(r^2-s^2)^{3/2}(1+c_1s^2+c_2s\sqrt{r^2-s^2})^3}(2c_1^2r^2s^3+c_2^2r^4s-c_2^2r^2s^3+2c_1r^2s\\
&+2c_1s^3+2s+(3c_1c_2r^2s^2+c_2(r^2+s^2)+c_2s^2)\sqrt{r^2-s^2})\\
=&\frac{r^2c_2\phi^3(2(c_1r^2+1)s+c_2r^2 \sqrt{r^2-s^2})}{4u(c_1r^2+1)\sqrt{r^2-s^2}^{3/2}(1+c_1s^2+c_2s\sqrt{r^2-s^2})^2},
\end{align*}

\begin{align*}
\Omega=&-\frac{3r^2c_2\phi^3}{4u(c_1r^2+1)(r^2-s^2)^{5/2}(1+c_1s^2+c_2s\sqrt{r^2-s^2})^3}(2c_1^2r^2s^3+c_2^2r^4s-c_2^2r^2s^3+2c_1r^2s\\
&+2c_1s^3+2s+(3c_1c_2r^2s^2+c_2(r^2+2s^2)+c_2s^2)\sqrt{r^2-s^2})\\
=&-\frac{3r^2c_2\phi^3(2(c_1r^2+1)s+c_2r^2 \sqrt{r^2-s^2})}{4u(c_1r^2+1)\sqrt{r^2-s^2}^{5/2}(1+c_1s^2+c_2s\sqrt{r^2-s^2})^2}.
\end{align*}
That is, we have easily that $m^2\Omega+3 \Psi=0$.

\medskip

Now, for simplicity and in order to evaluate the integral, we employ the following trick without loss of generality. 
We replace $c_1$ (resp.~$c_2$) by $\tfrac{c_1}{r^2}$ (resp.~$\tfrac{c_2}{r^2}$), so that we obtain

\begin{align*} 
\frac{c_1 s+c_2 \sqrt{r^2-s^2}}{r^2+c_1 s^2+c_2 s \sqrt{r^2-s^2}} =&\frac{1}{2} \frac{2 c_1 s \sqrt{r^2-s^2}-2 c_2 s^2+c_2 r^2}{\sqrt{r^2-s^2}\left(r^2+c_1 s^2+c_2 s \sqrt{r^2-s^2}\right)}\\
&+\frac{1}{2} \frac{c_2 r^2}{\sqrt{r^2-s^2}\left(r^2+c_1 s^2+c_2 s \sqrt{r^2-s^2}\right)}
\end{align*}

Then, we have 
\begin{align*}
\phi(r,s)&=\exp\left(\int \frac{c_1 s+c_2 \sqrt{r^2-s^2}}{r^2+c_1 s^2+c_2 s \sqrt{r^2-s^2}}ds\right) \\
&= \sqrt{r^2+c_1 s^2+c_2 s \sqrt{r^2-s^2}} e^{\frac{c_2}{\sqrt{c_2^2-4c_1-4}} \operatorname{arctanh}{\frac{c_2s+2\sqrt{r^2-s^2}}{s\sqrt{c_2^2-4c_1-4}}}}.
\end{align*}
\end{proof}
  
  \section{Landsberg uncorn's problem and $\sigma$T-condition }
  
   In this section, we investigate spherically symmetric Landsberg metrics satisfying the $\sigma T$-condition and derive the new solutions for the Landsberg unicorn's problem in Finsler geometry.

\begin{theorem}\label{Thm:Landsberg_1}
Let $n \geq 3$ and consider the spherically symmetric Finsler metric 
\(
F(x,y) = u\,\phi(r,s),  
\)
 and
\[
\phi(r,s) =
\sqrt{\,r^2 + c_1 s^2 + c_2 s \sqrt{r^2 - s^2}\,}\,
\exp\!\left(
\frac{c_2}{\sqrt{\,c_2^2 - 4c_1 - 4\,}}
\operatorname{arctanh}\!\left(
\frac{c_2 s + 2\sqrt{r^2 - s^2}}{s\sqrt{\,c_2^2 - 4c_1 - 4\,}}
\right)
\right),
\]
where $c_1, c_2 \in \mathbb{R}$ satisfy $c_2 \neq 0$ and $c_2^2 - 4c_1 - 4 \neq 0$.
Then $F$ is a non-regular Landsberg metric which is not Berwaldian.  
The associated spray coefficients are determined   by
\[
P = \frac{s}{r^2} + \frac{c_2 \sqrt{r^2 - s^2}}{r^2(c_1+1)}, 
\qquad
Q = \frac{c_1 - 1}{2r^2(c_1+1)} - \frac{c_1 s^2}{r^4(c_1+1)} - \frac{c_2 s \sqrt{r^2 - s^2}}{r^4(c_1+1)}.
\]
Moreover, $F$ reduces to a Berwald metric if and only if $c_2 = 0$.  
If $c_2^2 - 4c_1 - 4 < 0$, the expression involving $\operatorname{arctanh}$ is replaced by $\arctan$ using the identity
\(
I \cdot \operatorname{arctanh}(I z) = \arctan(z), \,  I := \sqrt{-1}.
\)
\end{theorem}
\begin{proof}
Consider
\[
\phi(r,s) =
\sqrt{\,r^2 + c_1 s^2 + c_2 s \sqrt{r^2 - s^2}\,}\,
\exp\!\left(
\frac{c_2}{\sqrt{c_2^2 - 4c_1 - 4}}
\operatorname{arctanh}\!\left(
\frac{c_2 s + 2\sqrt{r^2 - s^2}}{s\sqrt{c_2^2 - 4c_1 - 4}}
\right)
\right).
\]

The key idea is to compute explicitly the auxiliary functions \(P\) and \(Q\).
To this end, we first introduce
\[
T_s = \frac{\phi_s}{\phi}
   = \frac{c_1 s + c_2 \sqrt{r^2-s^2}}{r^2+c_1s^2+c_2 s\sqrt{r^2-s^2}}, 
   \qquad
T_r = \frac{\phi_r}{\phi}
   = \frac{r}{r^2+c_1s^2+c_2 s\sqrt{r^2-s^2}} .
\]
By differentiating, we obtain
\[
T_{ss}=\frac{\phi_{ss}}{\phi}
 = \frac{\partial T_s}{\partial s}+T_s^2
 = \frac{r^2\bigl(-c_2 s+c_1 \sqrt{r^2-s^2}\bigr)}{\sqrt{r^2-s^2}\,\bigl(r^2+c_1s^2+c_2 s\sqrt{r^2-s^2}\bigr)^2},
\]
\[
T_{rs}=\frac{\phi_{rs}}{\phi}
 = \frac{\partial T_s}{\partial r}+T_sT_r
 = -\frac{rs\bigl(-c_2 s+c_1 \sqrt{r^2-s^2}\bigr)}{\sqrt{r^2-s^2}\,\bigl(r^2+c_1s^2+c_2 s\sqrt{r^2-s^2}\bigr)^2}.
\]

Now, using the formulas for the geodesic spray coefficients, we get
\[
Q = \frac{1}{2r} \cdot \frac{-T_r + s T_{rs} + r T_{ss}}{1 - s T_s + (r^2 - s^2) T_{ss}}
   = \frac{c_1 - 1}{2r^2(c_1+1)} - \frac{c_1 s^2}{r^4(c_1+1)} - \frac{c_2 s \sqrt{r^2 - s^2}}{r^4(c_1+1)},
\]
\[
P = -\left( s   + (r^2 - s^2) T_s \right)Q
     + \frac{1}{2r } \left( s T_r + r T_s \right)
   = \frac{s}{r^2} + \frac{c_2 \sqrt{r^2 - s^2}}{r^2(c_1+1)}.
\]

Since \(n \geq 3\), Theorem~\ref{Theorem_A} implies that the metric is Landsbergian. 
To check whether it is Berwaldian, we compute
\[
P-sP_s=\frac{c_2}{(c_1+1)\sqrt{r^2-s^2}}, 
\qquad 
Q_s-sQ_{ss}=-\frac{c_2}{(c_1+1)(r^2-s^2)^{3/2}}.
\]
These expressions show that the Berwald conditions \eqref{Berwald_conditions} are not satisfied, 
so the metric is Landsbergian but not Berwaldian. 
\end{proof}

\medskip

\begin{theorem}\label{Thm:Landsberg_2}
Let $n \geq 3$ and consider the spherically symmetric Finsler metric 
\[
F(x,y) = u\,\phi(r,s), 
\]
where
\[
\phi(r,s)= 
\sqrt{\,r^2+c_1 s^2+c_2 s \sqrt{r^2-s^2}\,}\,
\exp\!\left(
\frac{c_2}{\sqrt{\,c_2^2-4c_1-4\,}} 
\operatorname{arctanh}\!\left(
\frac{c_2s+2\sqrt{r^2-s^2}}{s\sqrt{\,c_2^2-4c_1-4\,}}
\right)
\right).
\]
Then $F$ is non-regular  Landsbergian if and only if either  
\begin{enumerate}
\item[(i)] $c_1, c_2$ are constants, or  
\item[(ii)] $c_1 = b c_2^2 - 1$ for some constant $b\in\mathbb{R}$.  
\end{enumerate}

In case (ii), the metric takes the form
\[
\phi(r,s) =
\sqrt{\,r^2+(b c_2^2 -1)s^2 + c_2 s \sqrt{r^2-s^2}\,}\,
\exp\!\left(
\frac{1}{\sqrt{\,1-4b\,}}
\operatorname{arctanh}\!\left(
\frac{c_2 s + 2\sqrt{r^2-s^2}}{c_2 s \sqrt{\,1-4b\,}}
\right)
\right).
\]

Moreover, in this case the spray coefficients are given by the functions
\[
Q = -\frac{(b r c_2 c_2' - 2 b c_2^2 + 2) s^2}{2 b r^4 c_2^2}
     - \frac{s \sqrt{r^2-s^2}}{b r^4 c_2}
     + \frac{b c_2^2 - 2}{2 b r^2 c_2^2},
\quad
P = \frac{s}{r^2} + \frac{\sqrt{r^2-s^2}}{b r^2 c_2}.
\]

Hence $F$ is a Landsberg metric which is not Berwaldian.

If  $1-4b< 0$, the expression involving $\operatorname{arctanh}$ is replaced by $\arctan$ using the identity
\(
I \cdot \operatorname{arctanh}(I z) = \arctan(z).
\)
\end{theorem}

 \begin{proof}
Let 
\[
F(x,y) = u\,\phi(r,s), 
\]
be a spherically symmetric Finsler metric with $n \geq 3$, where 
\[
\phi(r,s)= \sqrt{r^2+c_1 s^2+c_2 s \sqrt{r^2-s^2}} 
   \exp\!\left(\frac{c_2}{\sqrt{c_2^2-4c_1-4}} 
   \operatorname{arctanh}\!\frac{c_2s+2\sqrt{r^2-s^2}}{s\sqrt{c_2^2-4c_1-4}}\right).
\]

If $c_1$ and $c_2$ are constants, then the result follows directly from Theorem~\ref{Thm:Landsberg_1}.  
Now suppose $c_1=c_1(r)$ and $c_2=c_2(r)$. A direct computation yields
\[
\frac{\phi_s}{\phi}=\frac{c_1 s+c_2 \sqrt{r^2-s^2}}{r^2+c_1s^2+c_2 s\sqrt{r^2-s^2}},
\]
and
\begin{align*}
\frac{\phi_r}{\phi}=&\frac{sr^2 c_2'-c_2's^3+c_1' s^2 \sqrt{r^2-s^2}+rc_2s+2r \sqrt{r^2-s^2}}{2\sqrt{r^2-s^2}(r^2+c_1s^2+c_2 s\sqrt{r^2-s^2})}\\
&+\frac{c_2s((c_2c_2'-2c_1')(r^2-s^2)-rc_2^2+4c_1 r+4r+(2c_1c_2'-c_1'c_2+2c_2')s\sqrt{r^2-s^2})}{2\sqrt{r^2-s^2}(r^2+c_1s^2+c_2s\sqrt{r^2-s^2})(c_2^2-4c_1-4)}\\
&-\frac{2(2c_1c_2'-c_1'c_2+2c_2')}{(c_2^2-4c_1-4)^{3/2}}\operatorname{arctanh}\left( \frac{c_2s+2\sqrt{r^2-s^2}}{s\sqrt{c_2^2-4c_1-4}} \right)\\
&=\frac{s  c_2'\sqrt{r^2-s^2}+c_1' s^2  +2r  }{2 (r^2+c_1s^2+c_2 s\sqrt{r^2-s^2})}+\frac{c_2s((c_2c_2'-2c_1')\sqrt{r^2-s^2}+(2c_1c_2'-c_1'c_2+2c_2')s)}{2(r^2+c_1s^2+c_2s\sqrt{r^2-s^2})(c_2^2-4c_1-4)}\\
&-\frac{2(2c_1c_2'-c_1'c_2+2c_2')}{(c_2^2-4c_1-4)^{3/2}}\operatorname{arctanh}\left( \frac{c_2s+2\sqrt{r^2-s^2}}{s\sqrt{c_2^2-4c_1-4}} \right).
\end{align*}

Comparing these expressions with the Landsberg conditions in \cite[Eq.~(5.2)]{Elgendi2021-SSM}, we see that the coefficient of the $\operatorname{arctanh}$ term must vanish. This leads to the ODE
\[
2c_1c_2'-c_1'c_2+2c_2'=0,
\]
whose solution is
\[
c_1 = b c_2^2 -1, \qquad b \in \mathbb{R}.
\]

Under this relation, the coefficients $P$ and $Q$ take the form
\[
P=\frac{s}{r^2}+\frac{\sqrt{r^2-s^2}}{b r^2 c_2}, 
\qquad
Q=-\frac{(-b r c_2c_2' +2b c_2^2 -2)s^2}{2b r^4 c_2^2}
   -\frac{s \sqrt{r^2-s^2}}{b r^4 c_2}
   +\frac{b c_2^2-2}{2b r^2 c_2^2}.
\]

Since $n \geq 3$, Theorem~\ref{Theorem_A} implies that the metric is Landsbergian. Moreover,
\[
P-sP_s=\frac{1}{b c_2 \sqrt{r^2-s^2}}, 
\qquad 
Q_s-sQ_{ss}=-\frac{1}{b c_2 (r^2-s^2)^{3/2}},
\]
which shows that the Berwald conditions \eqref{Berwald_conditions} do not hold.  
Hence the metric is Landsberg but not Berwald, completing the proof.
\end{proof}

 \begin{remark}
To construct further examples of Landsberg metrics that are not Berwaldian, we consider conformal transformations.  
By \cite{Elgendi-LBp}, a Landsberg metric $F$ remains Landsberg under the conformal change 
\[
\widetilde{F}(x,y)=e^{\sigma(x)}F(x,y)
\] 
if and only if the $\sigma T$-condition holds, namely 
\[
\sigma_h T^h_{ijk}=0,
\quad 
\sigma_h:=\frac{\partial \sigma}{\partial x^h},
\]
where $T^h_{ijk}$ is the T-tensor of $F$.  

Since the metric classes treated in Theorems \ref{Thm:Landsberg_1} and \ref{Thm:Landsberg_2} satisfy the $\sigma T$-condition, their conformal changes remain Landsberg but not Berwaldian under the same restrictions.  
In particular, for $n\geq 3$, the following families provide Landsberg non-Berwald metrics:
\[
\phi(r,s) = A(r)\,
\sqrt{\,r^2 + c_1 s^2 + c_2 s \sqrt{r^2 - s^2}\,}\,
\exp\!\left(
\frac{c_2}{\sqrt{c_2^2 - 4c_1 - 4}}
\operatorname{arctanh}\!\left(
\frac{c_2 s + 2\sqrt{r^2 - s^2}}{s\sqrt{c_2^2 - 4c_1 - 4}}
\right)
\right),
\]
and
\[
\phi(r,s) = B(r)\,
\sqrt{r^2+(bc_2^2-1) s^2+c_2 s \sqrt{r^2-s^2}}\,
\exp\!\left(
\frac{1}{\sqrt{1-4b}}
\operatorname{arctanh}\!\left(
\frac{c_2s+2\sqrt{r^2-s^2}}{c_2s\sqrt{1-4b}}
\right)
\right),
\]
subject to the restrictions given in Theorems \ref{Thm:Landsberg_1} and \ref{Thm:Landsberg_2}.
\end{remark}

\section*{Declarations}
\begin{itemize}
\item \textbf{Competing interests}: The author  declares no conflict of interest.
  \item \textbf{Availability of data and material}: Not applicable.
  \item \textbf{Funding}: Not applicable.
  \item \textbf{ Declaration of Generative AI and AI-assisted technologies in the writing process}: During the preparation of this work, the author used Gemini and ChatGPT to refine the language and improve it. After using these tools, the author  reviewed and edited the content as needed and take full responsibility for the content of the published article.
\end{itemize}

\providecommand{\bysame}{\leavevmode\hbox
to3em{\hrulefill}\thinspace}
\providecommand{\MR}{\relax\ifhmode\unskip\space\fi MR }
\providecommand{\MRhref}[2]{%
  \href{http://www.ams.org/mathscinet-getitem?mr=#1}{#2}
} \providecommand{\href}[2]{#2}

\end{document}